\documentclass[11pt, reqno]{amsart}
\usepackage{amssymb,latexsym,amsmath,amsfonts}
\usepackage{latexsym}
\usepackage[mathscr]{eucal}

\numberwithin{equation}{section}

\theoremstyle{definition}
\newtheorem{definition}{Definition}[section]

\theoremstyle{remark}
\newtheorem{remark}[definition]{Remark}
\theoremstyle{plain}
\newtheorem{theorem}[definition]{Theorem}
\newtheorem{lemma}[definition]{Lemma}

\newtheorem{result}[definition]{Result}

\newcommand{\eps}{\varepsilon}

\newcommand{\zbar}{\overline{z}}

\newcommand{\tht}{\theta}
\newcommand{\zahl}{\mathbb{Z}}

\newcommand{\bdy}{\partial}
\newcommand{\OM}{\Omega}

\newcommand{\Dsc}{\mathbb{D}}

\newcommand{\dom}{\mathcal{D}}

\newcommand{\smoo}{\mathcal{C}}
\newcommand{\hol}{\mathcal{O}}
\newcommand{\poly}{\mathscr{P}}

\newcommand{\bcdot}{\boldsymbol{\cdot}}
\newcommand{\rl}{{\sf Re}}

\newcommand{\impl}{\Longrightarrow}
\newcommand{\mapp}{\longmapsto}

\newcommand{\lrarw}{\longrightarrow}
\newcommand\partlD[2]{\partial_{{#1}}^{{#2}}}
\newcommand\ba[1]{\overline{#1}}
\newcommand\hull[1]{\widehat{#1}}

\newcommand{\leb}{\mathfrak{m}}

\newcommand{\CC}{\mathbb{C}^2}
\newcommand{\cplx}{\mathbb{C}}

\newcommand{\Gr}{{\sf Gr}}

\begin{document}

\title[Uniform algebras]{Uniform algebras generated by holomorphic \\
and close-to-harmonic functions}

\author{Gautam Bharali}
\address{Department of Mathematics, Indian Institute of Science, Bangalore -- 560012}
\email{bharali@math.iisc.ernet.in}
\author{Sushil Gorai}
\address{Department of Mathematics, Indian Institute of Science, Bangalore -- 560012}
\email{sushil@math.iisc.ernet.in}

\thanks{GB is supported by the DST via the Fast Track grant SR/FTP/MS-12/2007.
SG is supported by CSIR-UGC fellowship 09/079(2063). Support is also provided
by the UGC under DSA-SAP, Phase~IV}


\begin{abstract}
The initial motivation for this paper is to discuss a more concrete approach to an approximation theorem
of Axler and Shields, which says that the uniform algebra on the closed unit disc $\ba{\Dsc}$ generated
by $z$ and $h$ --- where $h$ is a nowhere-holomorphic harmonic function on $\Dsc$ that is continuous up
to $\bdy{\Dsc}$ --- equals $\smoo(\ba{\Dsc})$. The abstract tools used by Axler and Shields make 
harmonicity of $h$ an essential condition for their result. We use the concepts of
plurisubharmonicity and polynomial convexity to show that, in fact, the same conclusion is reached
if $h$ is replaced by $h+R$, where $R$ is a non-harmonic perturbation whose Laplacian is ``small'' in
a certain sense. 
\end{abstract} 
\maketitle

\section{Introduction and Statement of Results}\label{S:intro}

This paper is motivated by the following result of Axler and Shields \cite{AS} (in what follows, 
$\Dsc$ will denote the open unit disc in $\cplx$ centered at the origin):

\begin{result}[\cite{AS}, Theorem~4]\label{R:axlershields}
Let $h$ be a function in $\smoo(\ba{\Dsc})$ that is harmonic but nowhere holomorphic
on $\Dsc$. Then, $[z,h]_{\ba{\Dsc}} = \smoo(\ba{\Dsc})$.
\end{result}

\noindent{Recall that $[z,h]_{\ba{\Dsc}}$ denotes the uniform algebra on $\ba{\Dsc}$ generated by $z$ 
and $h$. Axler and Shields use results that are abstract and of extremely general scope --- such as
the Bishop Antisymmetric Decomposition --- to deduce their theorems. Harmonicity plays a very
central role in their approach, and it is difficult to answer even this simple question: {\em to
what extent can we allow harmonicity to fail, by adding a small perturbation $R$
to $h$, and yet recover the conclusion of Result~\ref{R:axlershields} with $h+R$ replacing $h$~?}}
\smallskip

Axler and Shields themselves imply that they tried to prove Result~\ref{R:axlershields} without 
the use of their deep result on the ${L}^\infty(\Dsc)$-subalgebra $H^\infty(\Dsc)[h]$ (which is where 
harmonicity plays a key role) but to no avail (see page~636 of \cite{AS} for their statement).
Hence, there is an interest in a more explicit approach even to Result~\ref{R:axlershields}.
\smallskip

We are able, using plurisubharmonic functions and polynomial convexity in a simple way, to prove 
an Axler-Shields-type result which states that $[z,h+R]_{\ba{\Dsc}}= \smoo(\ba{\Dsc})$, where $R$ is 
a small --- in an appropriate sense --- non-harmonic perturbation. Furthermore, taking $R=0$ in
our result reproduces the conclusion of Result~\ref{R:axlershields}, thus providing a different
approach to the Axler-Shields theorem.
\smallskip

\noindent{The central result of this article is:}

\begin{theorem}\label{T:genlaxler-shields}
Let $h:\ba{\Dsc}\lrarw\cplx$ be a function that is harmonic on $\Dsc$ and belongs to 
$\smoo(\ba{\Dsc})$. Let $R\in \smoo^2(\Dsc)\cap \smoo(\ba{\Dsc})$, and suppose $R$ is a non-harmonic 
perturbation of $h$ that is small in the following sense:
\begin{itemize}
\item[$a)$] The set $\{z\in\Dsc:\partial_{\zbar}(h+R)(z)=0\}$ has zero Lebesgue measure; and
\item[$b)$] The Laplacian of $R$ has the bound
\begin{equation}\label{E:cond2}
 |\Delta R(z)| \ \leq \ C~\frac{|\partial_{\zbar}(h+R)(z)|^2}{\sup_{\ba{\Dsc}}|h+R|}
 \;\; \forall z\in \Dsc,
\end{equation}
for some constant $C\in(0,1)$.
\end{itemize}
Then, $[z,h+R]_{\ba{\Dsc}}= \smoo(\ba{\Dsc})$.
\end{theorem}

\begin{remark}\label{R:recover}
Observe that we can recover Result~\ref{R:axlershields} from the above theorem. 
Firstly, $R\equiv 0$ certainly satisfies (b). Note, furthermore, that $\partlD{\zbar}{}h$ is 
anti-holomorphic on $\Dsc$. Its zeros in $\Dsc$ thus form a discrete subset of $\Dsc$, hence
a set of zero Lebesgue measure. All the hypotheses of Theorem~\ref{T:genlaxler-shields} are 
satisfied, and hence $[z,h]_{\ba{\Dsc}}= \smoo (\ba{\Dsc})$.
\end{remark}

Before proceeding to the proof, let us glance at the central ideas involved. The proof may be 
summarised as follows (in what follows, given a function $f$ and a set $S\subset{\sf dom}(f)$,
$\Gr_S(f)$ will denote the set ${\sf graph}(f)\cap(S\times\cplx)$, i.e. the portion of the graph 
of $f$ whose projection onto the first coordinate is $S$):
\begin{itemize}
\item We start with a construction that goes back to H{\"o}rmander and Wermer \cite{HW}: we define
the function
\[
 \qquad \psi_r(z,w):=|w-(h+R)(rz)|^2,
 \;\; (z,w) \in \ba{D(0;r^{-1})}\times\cplx, \; r\in(0,1), \notag
\]
which vanishes precisely on the graph of $(h+R)(r\bcdot)$. We use the condition \eqref{E:cond2}
to show that $\psi_r$ is plurisubharmonic in $\boldsymbol{\Delta}_r:=D(0;r^{-1})\times D(0;\rho)$,
where $\rho>0$ is large enough to contain the aforementioned graph.
\smallskip

\item From the last fact, and the fact that each $\boldsymbol{\Delta}_r$, $r\in (0,1)$, is Runge, we
realise that $\Gr_{\ba{\Dsc}}((h+R)(r\bcdot))$ is polynomially convex. But because
$(h+R)(r\bcdot)\lrarw (h+R)$ uniformly on $\ba{\Dsc}$ as $r\uparrow 1$, we deduce the same
for $\Gr_{\ba{\Dsc}}(h+R)$.
\smallskip

\item Knowing that $\Gr_{\ba{\Dsc}}(h+R)$ is polynomially convex, the first condition on $R$ allows
us to appeal to {\em a variation} on a theorem of Wermer \cite[Theorem 1]{W}. Wermer's original theorem
would have required us to demand that $h, R\in \smoo^1(\ba{\Dsc})$. However, with very slight 
modifications to Wermer's proof, we can appeal to the resulting theorem to infer that 
$[z,h+R]_{\ba{\Dsc}}= \smoo(\ba{\Dsc})$.
\end{itemize}
\smallskip

The main idea needed for the aforementioned variation on Wermer's theorem has been remarked upon
in \cite{W}.  However, it might be of interest to the reader to see the relevant lemmas carefully
restated to suit the present setting (i.e. with lower boundary regularity). Hence, we shall discuss
this variation in Section~\ref{S:technical}. The proof of Theorem~\ref{T:genlaxler-shields} will 
be presented in Section~\ref{S:genlaxler-shields}.
\smallskip

\noindent{{\bf{\em Added in proof:}} It was brought to our notice that stronger results subsuming
Result~\ref{R:axlershields} had been established by Chirka \cite{C} in 1969. However, there are 
gaps in the proofs of \cite[Theorem~4]{C} and \cite[Theorem~5]{C}, on which Chirka's results rely.
The most significant gap is the one in the proof of Theorem~4, which is false as stated. The proof
appears to presume that (in the notation of \cite[Theorem~4]{C}) $\widehat{K}_A$ is always
connected (assuming w.l.o.g. that $K$ is connected). In short: \cite{AS} has the earliest 
{\em complete} proof of Result~\ref{R:axlershields} that we are aware of. That said, we feel
that the {\em basic ideas} in \cite{C} could still (bypassing \cite[Theorem~4]{C} and
\cite[Theorem~5]{C} entirely) be made to work; thus recovering Chirka's results in $\CC$
(which are slightly stronger than ours).}
\medskip

\section{Technical Results}\label{S:technical} 

We begin this section with a technical, but essentially elementary, result. We must first explain
some notation. Given a domain $\OM\subset\cplx^d$ and a real-valued function $F\in\smoo^2(\OM)$, the
{\em Levi form of $F$ at $z$} --- denoted by $\mathscr{L}F(z;\bcdot)$ --- is the quadratic
form given by
\[
\mathscr{L}F(z;V) \ := \ 
\sum_{j,k=1}^d\frac{\partial^2F}{\partial z_j\partial\zbar_k}(z)v_j\overline{v}_k \quad
\forall V=(v_1,\dots,v_d)\in\cplx^d.
\] 
\smallskip

\begin{lemma} \label{L:leviform}
Let $\OM$ be a domain in $\cplx$ and let $f \in {\smoo}^2(\OM)$. Define the function $\psi(z,w):=|w-f(z)|^2, \
(z,w)\in \OM\times\cplx$. Then, for the Levi form
$\mathscr{L}\psi (z,w;V),~~ V=(V_1, V_2) \in \CC$, we have
\begin{equation} \label{E:leviform}
  \mathscr{L}\psi(z,w;V) \geq
\left(2\rl \left((\ba{f(z)-w})\partlD{z \zbar}{2}f(z)\right)+ |\partial_{\zbar}f(z)|^2 \right)|V_1|^2.
\end{equation}
In particular, if $f=h+R$, where $h \in{\sf harm}(\OM)$, we have
\begin{equation} \label{E:levi-harmonic}
\mathscr{L}\psi(z,w;V) \geq \left(|\partial_{\zbar}(h+R)(z)|^2 +
2\rl \left((\ba{h(z)+R(z)-w})\partlD{z \zbar}{2}R(z)\right)\right)|V_1|^2.
\end{equation}
\end{lemma}
\begin{proof}
 We compute:
\begin{align}
\partlD{z \zbar}{2}\psi(z,w) \ &= \ 2\rl(\partlD{z \zbar}{2}f(z).(\ba{f(z)-w}))+ |\partial_{z}f(z)|^2
                                 + |\partial_{\zbar}f(z)|^2,\notag\\
\partlD{z \ba{w}}{2}\psi(z,w) \ &= \ -\partial_{z}f(z), \notag\\
\partlD{w \ba{w}}{2}\psi(z,w) \ &= \  1.\notag
\end{align}
Now, using the above calculation, we have the Levi form $\mathscr{L}\psi (z,w;V)$, with
$V=(V_1, V_2) \in \CC$, as (we denote the function mapping $(z,w)\mapp w$ by $w$):
\begin{align}
 \mathscr{L}\psi(\bcdot;V)&= \left(2\rl (\partlD{z \zbar}{2}f\bcdot(\ba{f-w}))+ |\partial_{z}f|^2+
	|\partial_{\zbar}f|^2 \right)|V_1|^2 - 2\rl(\partial_{z}f\bcdot V_1 \ba{V_2})+ |V_2|^2 \notag\\
&= |\partial_{z}f\bcdot V_1-V_2|^2+ \left(2\rl(\partlD{z \zbar}{2}f\bcdot(\ba{f-w}))+
	|\partial_{\zbar}f|^2 \right)|V_1|^2 \notag\\
&\geq \left( 2\rl(\partlD{z \zbar}{2}f \bcdot (\ba{f-w}))+ |\partial_{\zbar}f|^2 \right)|V_1|^2. \notag
\end{align}
The second inequality follows by replacing $f$ by $h+R$ and noting that $\partlD{z \zbar}{2}h=0$.
\end{proof}
\smallskip

We now present the following variation on \cite[Theorem 1]{W}. We need to clarify some notation
needed in its proof: given a compact subset $K\Subset \cplx^d$, we define
\[
\poly(K) \ := \ \text{the class of uniform limits on $K$ of holomorphic polynomials in $\cplx^d$.}
\] 

\begin{theorem}[A variation on Theorem~1 of \cite{W}]\label{T:Wermer}
Let $f:\ba{\Dsc}\lrarw\cplx$ be a continuous function. Assume $\Gr_{\ba{\Dsc}}(f)$ is polynomially
convex. Define
\begin{align}
 \mathcal{S}:=\{z\in\Dsc: \, &\text{there exists a $\cplx$-open neighbourhood $V_z\ni z$ such that} 
 \notag \\
			&\text{$f$ has continuous first-order partial derivatives on $V_z$}\}, \notag
\end{align}
and let $W:=\{z\in\mathcal{S}:\partlD{\zbar}{}f(z)\neq 0\}$. If $\ba{\Dsc}\setminus W$ has zero
Lebesgue measure, then $[z,f]_{\ba{\Dsc}}=\smoo(\ba{\Dsc})$.
\end{theorem}

\begin{remark} Wermer's original result requires that $f\in\smoo^{1}(\ba{\Dsc})$ and that
$\partlD{\zbar}{}f$ be non-vanishing. But immediately after the proof of \cite[Theorem~1]{W}, 
it is stated that the hypothesis of \cite[Theorem~1]{W} can 
be weakened by letting the condition $\partlD{\zbar}{}f(z)\neq 0$ 
fail on a non-empty subset of zero Lebesgue measure. The key to Theorem~\ref{T:Wermer}
is that $f$ can also be allowed to be non-differentiable on this exceptional set. We
justify this below.
\end{remark}  

\noindent{{\em Sketch of the proof of Theorem~\ref{T:Wermer}.}
Since the proof involves minor modifications to the original, we shall be brief.
The notation $\mu\perp [z,f]_{\ba{\Dsc}}$ will denote a complex measure 
$\mu\in\smoo(\ba{\Dsc})^\star$ representing a bounded linear functional on $\smoo(\ba{\Dsc})$ that
annihilates $[z,f]_{\ba{\Dsc}}$.}
\smallskip

Wermer's proof uses the following fact, which occurs as a part of Bishop's proof 
of \cite[Theorem~4]{B}:
\smallskip 

\begin{itemize}
\item[$(*)$] (Bishop) {\em For any complex measure $\mu\in\smoo(\ba{\Dsc})^\star$, let
\[
 H_\mu(a) \ := \ \int_\cplx \frac{1}{z-a}d\mu(z).
\]
If $H_\mu=0$ $\leb$-a.e., then $\mu=0$ ($\leb$ denotes the planar Lebesgue measure on $\cplx$).}
\end{itemize}

\noindent{Wermer's strategy consists of the following two parts:
\begin{enumerate}
\item[(a)] Use the Oka-Weil theorem
to construct, for each $a\in\ba{\Dsc}$, an open neighbourhood 
$\dom\supset \Gr_{\ba{\Dsc}}(f)$ and a function 
$h\in \hol(\dom)$ (which depend on $a$) such that:
\begin{itemize}
\item $\exists R>0$ such that $h(\Gr_{\ba{\Dsc}}(f))\subset E_R\cup \{0\}$, where $E_R$ denotes
the complement in $\cplx$ of the closed sector $\{re^{i\tht}:0\leq r\leq R, \ |\tht|\leq \pi/4\}$.
\item $h$ is non-vanishing on $\Gr_{\ba{\Dsc}}(f)\setminus\{(a,f(a))\}$.
\item $\exists h_1\in \hol(\dom)$ satisfying $h(z,w)=(z-a)h_1(z,w) \; \forall (z,w)\in\dom$.
\end{itemize}

\item[(b)] Apply (for a fixed $a\in\ba{\Dsc}$) the dominated convergence theorem to the
measure $\mu\perp [z,f]_{\ba{\Dsc}}$ and the sequence 
\[
 \left\{h_1(\bcdot \, ,f)P_n\circ h(\bcdot \, ,f)\right\}_{n\in\zahl_+}
 \subset\poly(\ba{\Dsc})
\]
(where the sequence $\{P_n\}_{n\in\zahl_+}$ is as given by \cite[Lemma~3]{W}) to conclude
that $H_\mu(a)=0$. This implies $H_\mu\equiv 0$ because the above argument works for each
$a\in\ba{\Dsc}$. 
\end{enumerate}}

\noindent{Since $H_\mu\equiv 0$ for {\em every} $\mu\perp [z,f]_{\ba{\Dsc}}$, 
it follows from $(*)$ that $[z,f]_{\ba{\Dsc}}=\smoo({\ba{\Dsc}})$.}
\smallskip

Observe that the inference $H_\mu\equiv 0$ (for $\mu\perp [z,f]_{\ba{\Dsc}}$) is 
stronger than is necessary for the desired conclusion. This suggests the 
following {\em modified} two-step strategy:
\begin{itemize}
\item[(a${}^\prime$)] Construct the objects $(\dom, h, E_R, h_1)$ having exactly
the same properties as in 
Part~(a) above, but {\em only associated to each $a\in W$}.

\item[(b${}^\prime$)] Repeat Part~(b) of Wermer's strategy for all those points
$a\in \ba{\Dsc}$ for which Wermer's dominated-convergence-theorem argument,
showing $H_\mu(a)=0$, still makes sense (call the complement of all such points $\mathcal{E}$).
\end{itemize}

\noindent{It is not hard to see that 
$\mathcal{E}=(\ba{\Dsc}\setminus W)\cup (\widetilde{E}\cap \ba{\Dsc})$, where:
\[
 \widetilde{E}:=\left\{a\in\cplx: \int_\cplx |z-a|^{-1}d|\mu|(z)=\infty\right\}.
\]
The set $\widetilde{E}$ has zero Lebesgue measure, which is a well-known fact about
finite, positive Borel measures in general. By {\em exactly the same} considerations as 
in Part~(b) --- and from the obvious
fact that $\mu\perp [z,f]_{\ba{\Dsc}}\impl H_\mu(a)=0 \ \forall a\notin \ba{\Dsc}$ --- our
modified strategy gives us}
\[
\mu\perp [z,f]_{\ba{\Dsc}} \; \; \impl \; \; H_\mu(a)=0 \; \forall a\in 
\cplx\setminus \mathcal{E}.
\]
Since, by hypothesis, $\leb(\mathcal{E})=0$, we infer from $(*)$ that each
$\mu\perp [z,f]_{\ba{\Dsc}}$ is just the zero measure.  
Hence, $[z,f]_{\ba{\Dsc}}=\smoo(\ba{\Dsc})$. \hfill $\Box$
\medskip

\section{The proof of Theorem \ref{T:genlaxler-shields}}\label{S:genlaxler-shields}

We recall a standard notation that we shall use in our proof. Given a domain 
$\OM\subset \cplx^d$ and a compact subset $K\Subset \OM$, we define the 
{\em $\hol(\OM)$-hull of $K$} as
\[
\hull{K}_\OM \ := \ \{z\in\OM: |f(z)|\leq \sup\nolimits_{K}|f| \; \forall f\in\hol(\OM)\}.
\]

\begin{proof}[Proof of Theorem~\ref{T:genlaxler-shields}]
We begin with a preliminary observation. The estimate \eqref{E:cond2} may be rewritten as
\[
 |\Delta R(z)| \ \leq \
\frac{|\partial_{\zbar}(h+R)(z)|^2}{\sup_{\ba{\Dsc}}|h+R|+\left(\frac{1}{C}-1\right)\sup_{\ba{\Dsc}}|h+R|},
\]
whence we can certainly find a constant $\delta_0>0$ such that
\[
\sup_{\ba{\Dsc}}|h+R|+\left(\tfrac{1}{C}-1\right)\sup_{\ba{\Dsc}}|h+R|
\ \geq \ \sup_{\ba{\Dsc}}|h+R|+\delta_0.
\]
Hence, for the remainder of this proof, we may assume that
\begin{equation}\label{E:cond1}
 |\Delta R(z)| \ \leq \
\frac{|\partial_{\zbar}(h+R)(z)|^2}{\sup_{\ba{\Dsc}}|h+R|+\delta_0}~~~~\forall z\in\Dsc.
\end{equation}

For each $r\in (0,1)$, let us define
\[
 \psi_r(z,w):=|w-(h+R)(rz)|^2, ~~~(z,w) \in\ba{D(0;r^{-1})}\times \cplx.
\]
The Levi-form computations \eqref{E:leviform} and \eqref{E:levi-harmonic}, taken together
with the the estimate \eqref{E:cond1} on $\Delta R$, establish that
\begin{equation}\label{E:psh}
\psi_r \ \text{\em is plurisubharmonic in $\boldsymbol{\Delta}_r:=D(0;r^{-1})\times D(0;M+2\delta_0)$
$\forall r\in (0,1)$},
\end{equation}
where $M:= \sup_{\ba{\Dsc}}|h+R|$. Given these preliminaries, we can complete the proof in
two steps.
\smallskip

\noindent{{\bf Step I:} {\em Polynomial convexity of $\Gr_{\ba{\Dsc}}(h+R)$.}}
\smallskip

\noindent{Since $(h+R)$ is uniformly continuous on $\ba{\Dsc}$, it follows that:
\begin{align}
\text{\em For each $\eps>0$}, \ \exists&\delta(\eps)>0 \ \text{\em such that} \notag \\
0<(1-r) \ \leq \ &\delta(\eps) \ \Longrightarrow \ |(h+R)(rz)-(h+R)(z)| < \eps \;\; 
\forall z\in \ba{\Dsc}. \label{E:trap}
\end{align}
Consider a point 
$p=(z_0,w_0)\in \ba{\Dsc}\times D(0;M+2\delta_0)\setminus \Gr_{\ba{\Dsc}}(h+R)$. Then,
by definition, $\psi_1(z_0,w_0)=:\eps_p^2>0$. Write $r(p):=1-\delta(\eps_p/3)$, where
$\delta(\eps_p/3)$ is as given by \eqref{E:trap}. Then, \eqref{E:trap} tells us:
\begin{align}
|w_0-(h+R)(r(p)z_0)|  &\geq  |w_0-(h+R)(z_0)|-|(h+R)(z_0)-(h+R)(r(p)z_0)| \notag\\
			    &>  2\eps_p/3, \notag\\ 
\psi_{r(p)}(z,w)  &=  |(h+R)(z)-(h+R)(r(p)z)|^2 < \eps_p^2/9 \;\;
\forall (z,w)\in \Gr_{\ba{\Dsc}}(h+R). \notag
\end{align}    
By the last two estimates, we have just shown that 
\begin{equation}\label{E:sep}
\psi_{r(p)}(p) > 4\eps_p^2/9 \; \; \; \text{and} \; \; \; \psi_{r(p)}(z,w)\ < \eps_p^2/9 \;
\forall (z,w)\in \Gr_{\ba{\Dsc}}(h+R).
\end{equation}
Let us now write $K:=\Gr_{\ba{\Dsc}}(h+R)$. We claim that $p\notin\hull{K}$. To do so,
we invoke a well-known result of H\"{o}rmander~\cite[Theorem~4.3.4]{H} which states that if 
$\OM\subset\cplx^d, \ d\geq 2$, is a pseudoconvex domain and $K\Subset \OM$ is a compact
subset, then the hull $\hull{K}_\OM$ can also be expressed as:
\[
\hull{K}_\OM \ = \ \{z\in \cplx^d : U(z) \leq  \sup\nolimits_{K}U \;\;
\forall U\in {\sf psh}(\OM)\}.
\]
It thus follows from \eqref{E:sep} that
$p\notin \hull{K}_{\boldsymbol{\Delta}_{r(p)}}$. Note that each $\boldsymbol{\Delta}_{r(p)}$ 
is Runge. We know therefore that $\hull{K}_{\boldsymbol{\Delta}_{r(p)}}=\hull{K}$.
Since $p$ was arbitrarily chosen, we have just shown that
\[
p\notin \hull{K} \; \; \forall p\in \ba{\Dsc}\times D(0;M+2\delta_0)\setminus K.
\]
Of course, it is easy to see that no point in $\CC\setminus (\ba{\Dsc}\times\ba{D(0;M)})$ 
can belong to $\hull{K}$. Hence, $K=\Gr_{\ba{\Dsc}}(h+R)$ is polynomially convex.}
\smallskip

\noindent{\bf Step II:} {\em Completing the proof.}
\smallskip

\noindent{We appeal to Theorem~\ref{T:Wermer} with $(h+R)$ playing the role of $f$. In the terminology
of Theorem~\ref{T:Wermer}
\[
\ba{\Dsc}\setminus W \ = \ \bdy\Dsc\cup \{z\in \Dsc: \partlD{\zbar}{}(h+R)(z)=0\},
\]
which, by hypothesis, has zero Lebesgue measure. We have already established that
$\Gr_{\ba{\Dsc}}(h+R)$ is polynomially convex. Thus, $(h+R)$ satisfies all the conditions 
stated in Theorem~\ref{T:Wermer}, and we conclude that 
$[z,h+R]_{\ba{\Dsc}}= \smoo (\ba{\Dsc})$.}
\end{proof}
\medskip

\noindent{{\bf Acknowledgements.} We thank the referee of this article for pointing out 
ways in which the arguments in our first draft of this article could be simplified, and for
suggesting that we examine whether the boundary regularity of the functions in 
Theorem~\ref{T:genlaxler-shields} could be lowered.}

\end{document}